\documentclass[12pt]{amsart}

\usepackage{amsmath,amssymb,amsfonts}

\newtheorem{theorem}{Theorem}[section]
\newtheorem{lemma}[theorem]{Lemma}
\newtheorem{proposition}[theorem]{Proposition}

\newtheorem{thm}{Theorem}
\usepackage{xcolor}

\begin{document}

\title[One-generator nilpotent braces]{On the structure of some one-generator nilpotent braces}

\author{Martyn R. Dixon}
\address[Martyn Dixon]
{Department of Mathematics\\
University of Alabama\\
Tuscaloosa, AL 35487-0350, U.S.A.}
\email{mdixon@ua.edu}
\author{Leonid A. Kurdachenko}
\address[Leonid Kurdachenko]
{Department of Algebra, Facultet of mathematic and mechanik\\
National University of Dnepropetrovsk\\
Gagarin prospect 72\\
Dnipro 10, 49010, Ukraine.}
\email{lkurdachenko@i.ua, lkurdachenko@gmail.com}
\author[I. Subbotin]{Igor Ya. Subbotin}
\address[Igor Ya. Subbotin]{Department of Mathematics and Natural Sciences, National University\\
5245 Pacific Concourse Drive,  \\
Los Angeles, CA 90045-6904, USA.}
\email{isubboti@nu.edu}

\begin{abstract} 
This article provides a detailed description of some nilpotent left braces generated by one element.
\end{abstract} 
\thanks{The second author would like to thank the Isaac Newton Institute for Mathematical Sciences, 
Cambridge, for support and hospitality during the Solidarity Supplementary Grant Program 
when work on this paper was undertaken. This work was supported by 
EPSRC grant no EP/R014604/1.34. The second author is grateful to the Head of the School of Mathematic at the University of Edinburgh,  Bernd Schoers,  and the personal Chair in Algebra, Agata Smoktunowicz.
}
\dedicatory{To the memory of our friend, Francesco de Giovanni}



\maketitle

\newcommand{\sub}[2]{\langle#1,#2\rangle}
\newcommand{\cyclic}[1]{\langle #1\rangle}
\newcommand{\norm}{\triangleleft\,}
\newcommand{\core}[2]{\text{core}\,_{#1}\,#2}
\newcommand{\cpin}{C_{p^{\infty}}}
\newcommand{\co}[1]{\text{$\mathbf{#1}$}}
\newcommand{\cogo}[2]{\co{#1}\mathfrak{#2}}
\newcommand{\ann}{\text{Ann\,}}
\newcommand{\spec}{\text{Spec\,}}
\newcommand{\listel}[2]{#1_1,\dots, #1_{#2}}
\newcommand{\aut}{\text{Aut}\,}
\newcommand{\cent}[2]{C_{#1}(#2)}
\newcommand{\Dr}{\text{Dr}\,}
\newcommand{\dir}[3]{\underset{#1\in #2}{\Dr}#3_{#1}}
\newcommand{\nocl}[2]{\langle #1\rangle^{#2}}
\newcommand{\qn}{\,\text{qn}\,}
\newcommand{\comm}[2]{#1^{-1}#2^{-1}#1#2}
\newcommand{\conj}[2]{#1^{-1}#2#1}
\newcommand{\izer}[2]{N_{#1}(#2)}
\newcommand{\gl}{\text{GL\,}}
\newcommand{\subgp}[3]{\langle#1_{#2}\mid #2\in#3\rangle}
\newcommand{\direct}[3]{\underset{#1\in #2}{\Dr}\langle#3_{#1}\rangle}
\newcommand{\omn}[1]{\Omega_n{(#1)}}
\newcommand{\Hom}{\text{Hom}\,}
\newcommand{\Cr}[3]{\underset{#1\in #2}{\text{Cr}}\,#3_{#1}}
\newcommand{\inflist}[1]{\{#1_1,#1_2,\dots\}} 
\newcommand{\cart}{\text{Cr}\,} 
\newcommand{\carti}[3]{\underset{#1\geq #2}\cart #3_{#1}} 
\newcommand{\fgsub}[2]{\langle #1_1,\dots,#1_{#2}\rangle}
\newcommand{\cwr}{\,\bar{\wr}\,}
\newcommand{\minn}{\infty$-$\overline{\mathfrak{N}}}
\newcommand{\mnn}{\overline{\mathfrak{N}}}
\newcommand{\mnp}{\overline{\mathcal{P}}}
\newcommand{\mns}{\overline{\mathfrak{S}}}
\newcommand{\minsd}{\infty$-$\overline{\mathfrak{S}_d}}
\newcommand{\minnc}{\infty$-$\overline{\mathfrak{N}_c}}
\newcommand{\ms}{\mathfrak{S}}
\newcommand{\mn}{\mathfrak{N}}
\newcommand{\mf}[1]{\mathfrak{#1}}
\newcommand{\sdr}{\mathfrak{S}_d(r)}
\newcommand{\sd}{\mathfrak{S}_d}
\newcommand{\mfr}{\mathfrak{R}}
\newcommand{\ncr}{\mathfrak{N}_c(r)}
\newcommand{\mfsr}{\overline{\ms\mfr}}
\newcommand{\lslf}{\overline{(\cogo{L}{S})(\cogo{L}{F})}}
\newcommand{\mcp}{\mathcal{P}}
\newcommand{\mcps}{\mathcal{P}^*}
\newcommand{\mfas}{\mathfrak{A}^*}
\newcommand{\mfa}{\mathfrak{A}}
\newcommand{\mfss}{\mathfrak{S}^*}
\newcommand{\mfns}{\mathfrak{N}^*}
\newcommand{\wmcnq}{max-$\infty$-$\overline{\text{qn}}$ }
\newcommand{\om}{\omega}
\newcommand{\bog}[1]{\bar{\omega}(#1)}
\newcommand{\og}[1]{\omega(#1)}
\newcommand{\wi}{\overline{\om}_i}
\newcommand{\w}{\overline{\om}}
\newcommand{\fsn}{\emph{f}-subnormal }
\newcommand{\wiw}{\om_i}
\newcommand{\ww}{\om}

\newcommand{\zefin}[1]{\zeta_{\infty}(\star,#1)}
\newcommand{\zet}[1]{\zeta(\star,#1)}
\newcommand{\zen}[2]{\zeta_{#1}(\star,#2)}
\newcommand{\al}{\alpha}
\newcommand{\ga}{\gamma}
\newcommand{\be}{\beta}
\newcommand{\la}{\lambda}
\newcommand{\La}{\Lambda}
\newcommand{\ze}{\zeta}
\newcommand{\soc}[1]{\text{Soc}(#1)}
\newcommand{\zl}[1]{\text{zl}(#1)}
\newcommand{\ldl}{\leq\dots\leq}
\newcommand{\br}[1]{\textbf{br}(#1)}
\newcommand{\oq}[1]{(#1_1,#1_2,#1_3,#1_4)}

\section{Introduction} \label{s:intro}

A \emph{left brace  (or skew left brace of abelian type)}  is a set $A$ together with two binary operations, addition denoted by $+$ and multiplication denoted by $\cdot$ (which is often omitted), satisfying the following conditions:

\begin{enumerate}
\item[LB1] $A$ is an abelian group under addition;
\item[LB2] $A$ is a group under multiplication;
\item[LB3] $a(b+c)= ab+ac -a$ for all $a, b, c\in  A$.
\end{enumerate}

In this paper, all braces will be left braces in this sense.
Braces were first introduced by W. Rump in \cite{wR05} as a generalization of Jacobson radical rings in order to help study involutive set-theoretic solutions of the Yang-Baxter equation.  
This fundamental equation has profound implications in both pure mathematics and physics.  It originated from the groundbreaking work of the Nobel prize-winning physicist Yang in the realm of statistical mechanics and independently in the contributions of Baxter to the $8$-vertex model. This theory holds substantial significance across diverse domains such as knot theory, braid theory, operator theory, Hopf algebras, quantum groups, $3$-manifolds and the monodromy of differential equations.

In some sense braces represent a type of fusion of two groups  such that the zero element of the additive group is the identity element of the multiplicative group and we denote this common neutral element by $0$ or $1$, as appropriate.  There is therefore a fundamental difference between left braces and associative rings.  We shall sometimes denote the additive group of $A$ by $(A,+)$ and the multiplicative group of $A$ by $(A,\cdot)$.  

The theory of left braces is currently developing rapidly and as with other branches of algebra the examination of the internal structure of braces plays a crucial role in their study. 

One of the first natural steps in the study of an algebraic structure is to clarify the structure of its substructures generated by one element.  For some structures (for example,  groups, Lie algebras, commutative rings) this task is simple. However, for left braces this is generally a highly complex problem. Therefore, it is natural to begin studying the structure of such left braces under some additional conditions. This work is devoted to the study of one-generator left braces. 

We first introduce some necessary concepts.
 The reader is referred to \cite{fC18} for some elementary properties and the connections with Yang-Baxter equations, which we do not address here.  
 
 
 If $A$ is a left brace, then a subset $S$ of $A$ is called a \emph{subbrace} (more precisely a \emph{left subbrace}) 
 if $S$ is closed under  addition and multiplication and is a left brace by 
restriction of these operations to  $S$.  Thus a subset $S$ of $A$ is a subbrace of $A$ if and only if $S$ is closed under addition and multiplication  and under addition $S$ is an additive subgroup of $A$ and under multiplication $S$ is a multiplicative subgroup of $A$.
 
 For the left brace $A$ and for each element $a\in A$ we define a map $\lambda_a:A\longrightarrow A$ by $\lambda_a(x)=ax-a$ for all $x\in A$. Furthermore we define $a\star b =ab-a-b$ for all $a,b\in A$. It is easy to see that $a\star b=\lambda_a(b)-b$. The operation $\star$  plays a very important role in the study of left braces.
 
 A left brace $A$ is called \emph{trivial} or \emph{abelian} if $a\star b=0$ for all $a,b\in A$, or equivalently, $a+b=ab$.  In the case of abelian braces the addition and multiplication coincide.
 
A subbrace $L$ of the left brace $A$ is called a \emph{left ideal} of $A$ if $a\star b\in L$ for all elements $a\in A$ and $b\in L$.  
This means that if $x, y\in L$, then $x-y\in L$ and if $a\in A$ and $z\in L$, then $\la_a(z)\in L$,
as  in \cite{fC18}.
 We recall that a subbrace $L$ of a left brace $A$ is an \emph{ideal} if $a\star z, z\star a\in L$ for all $a\in A, z\in L$.  
 
 If $A$ is a left brace and $K,L$ are subbraces, then we let $K\star L$ denote the subgroup of $(A,+)$ generated by the elements $x\star y$, where $x\in K, y\in L$.
 
 We define two  canonical series of the left brace $A$.  First, let $A^{(1)}=A$ and recursively define $A^{(\alpha +1)}=A^{(\alpha)}\star A$ for all ordinals $\alpha$ and as usual set $A^{(\lambda)}=\cap_{\mu<\lambda}A^{(\mu)}$ for all limit ordinals $\lambda$.  As in \cite{CGS17} $A$ is called \emph{right nilpotent} if $A^{(n)}=0$ for some natural number $n$.
 
 Similarly, let $A^{1}=A$, $A^{\alpha+1}=A\star A^{\alpha}$, for all ordinals $\alpha$ and $A^{\lambda}=\cap_{\mu<\lambda}A^{\mu}$ for all limit ordinals $\lambda$. As in \cite{CGS17} $A$ is \emph{left nilpotent} if $A^{m}=0$ for some natural number $m$. We note that a left nilpotent brace needs not be right nilpotent and vice-versa (see \cite{dB15} and \cite{fC18}, for example). Indeed Rump~\cite{wR07b} gave an example of a left brace of cardinality $6$ such that $A^{(3)}=0$, but $A^m\neq 0$ for all $m$.
 
 It is shown in \cite{fC18} that $A^n$ is always a left ideal of $A$ and $A^{(n)}$ is an ideal of $A$. This has been extended in \cite{DKS24} to include all ordinals $\alpha$.

 In the theory of braces there are several different approaches to the concept of nilpotency and we refer the reader to the papers \cite{fC18, CGS17, CSV19, JVV22,aS18,aS22}.  
 Our approach to nilpotency is based on the following concept, first introduced by Bonatto and Jedli\v{c}ka~\cite{BJ21} 
 and Jespers, Van Antwerpen and Vendramin~\cite{JVV22}. 
 
 Let $A$ be a left brace. The set 
 \begin{align*}
 \zeta(\star, A)&=\{a\in A| a\star x=x\star a= 0 \text{ for all }x\in A\}\\
 &=\{a\in A| ax=a+x=xa \text{ for all } x\in A\}\\
 &=\{a\in A| \la_a(x)=x \text{ and } \la_x(a)=a \text{ for all }x\in A\}
 \end{align*}
 is here called the \emph{$\star$-center} of $A$. It is well known  that $\zet{A}$ is an ideal of $A$ which is contained in the center, $\zeta(A)$, of the multiplicative group of $A$, a  result which appears in \cite{BJ21} and \cite{JVV22}.

 By analogy with groups we construct the \emph{upper $\star$-central series}
 \[
 0=\zeta_0(\star, A)\leq \zeta_1(\star, A)\leq \dots \zeta_{\alpha}(\star, A)\leq \zeta_{\alpha+1}(\star, A) \leq 
\dots \zeta_{\gamma}(\star, A)
\]
 as follows. We let 
 \begin{align*}
 \zeta_{1}(\star, A)&=\zeta(\star, A) \text{ and }\\
 \zeta_{\alpha+1}(\star, A)/\zeta_{\alpha}(\star, A)&=\zeta(\star, A/\zeta_{\alpha}(\star, A))
 \end{align*}
  for all ordinals $\alpha$; as usual for limit ordinals $\lambda$ we set $\zeta_{\lambda}(\star, A)=\cup_{\mu<\lambda}\zeta_{\mu}(\star, A)$.

 This definition and our remark above imply that each term of this series is an ideal of $A$.  The last term $\zeta_{\infty}(\star, A)=\zeta_{\gamma}(\star, A)$ of this series is called the  \emph{upper $\star$-hypercenter} of $A$ and we denote the length of the upper $\star$-central series of $A$ by $\text{zl}(A)$. Furthermore, if $A=\zeta_{\infty}(\star, A)$, then $A$ is called a \emph{$\star$-hypercentral} brace; in this case, if $\text{zl}(A)$ is finite, then we say that $A$ is \emph{$\star$-nilpotent}.  This series has also been introduced in \cite{BJ21, CFT23, mT23}. 
 
 
 When $A$ is a left brace and $B,C$ are ideals of $A$ with $B\leq C$,  then the factor $C/B$ is called \emph{$\star$-central} if $A\star C, C\star A\leq B$. Thus for all $a\in A, c\in C$ we have 
 $a\star c, c\star a \in B$.
 

 We shall say that $A$ is \emph{Smoktunowicz-nilpotent} if there are natural numbers $n,k$ such that $A^{(n)}=A^{k}$. Such braces were introduced by A. Smoktunowicz in the paper \cite{aS18}. In \cite[Theorem 1.3]{aS18} a single criterion was given for these conditions and left braces satisfying this single criterion were in \cite{aS22} called \emph{strongly nilpotent}.

A  left brace is $\star$-nilpotent if and only if it is Smoktunowicz-nilpotent (see \cite{DKS24} for example). 
 In this paper we study the structure of certain one-generator Smoktunowicz-nilpotent braces.
 Here,  if   $A$ is a left brace and $M$ is a subset of $A$, then the \emph{subbrace $B$ of $A$ generated by $M$} is the intersection of all subbraces $C$ of $A$ that contain $M$.  
 If $M$ is a finite set, then $B$ will be called \emph{finitely generated}. We let $\textbf{br}(M)$ denote the subbrace generated by $M$. We refer the reader to \cite {CFT23,JVV22} for further details and properties.
 
 We shall let $\mathfrak{N}_S(n,k)$ denote the class of left braces $A$ satisfying $A^{(n)}=0=A^k$ where $n,k$ are the least integers with this property.
 
 One-generator braces $A$ satisfying $A^3=0$ were studied in the paper \cite{KS24} and one-generator braces satistfying $A^{(3)}=0$ were studied in the paper \cite{BEKP24}. The next natural step is the study of left braces in the class $\mathfrak{N}_S(4,4)$. The computations here become very cumbersome. We note that left braces in this class are $\star$-nilpotent of $\star$-nilpotency class at most $16$.  Therefore,  we have  changed the approach a little and moved on to a consistent study of one-generator $\star$-nilpotent braces and such a study begins in this article. 
 
 The first natural step here is the study of a one-generator left brace $A$ such that $\zl{A}=2$. The description of such a brace is known, but we have not found a source where this is given so for the sake of completeness we present this here. 
 
 Let $D(1,2)= \mathbb{Z}\times \mathbb{Z}$ and define an addition and a multiplication on $D(1,2)$ by 
\begin{align}\label{eq:defD12}
(m_1,m_2)+(n_1,n_2)&=(m_1+n_1,m_2+n_2)\\
(m_1,m_2)\cdot (n_1,n_2) &=(m_1+n_1,m_2+n_2+m_1n_1)
\end{align}
for $m_i,n_i\in \mathbb{Z}$.
Our first main result is as follows.

\begin{thm}\label{t:maina}
\begin{enumerate}
\item[ (i)] The set $D(1,2)$ with the addition and multiplication as defined is a one-generator, Smoktunowicz-nilpotent left brace and $\zl{D(1,2)}=2$.
\item[(ii)] Let $A$ be a non-abelian one generator left brace. If $A=\zen{2}{A}$, then $A$ is an epimorphic image of $D(1,2)$.
\end{enumerate}
\end{thm}

The main part of this work is devoted to the study of periodic one-generator $\star$-nilpotent left braces $A$ such that $\zl{A}=3$. Our second main result is:

\begin{thm}\label{t:mainb}
Let $A$ be a one-generator 
left brace, generated by an element $a$ which satisfies $na=0$ for some natural number $n$. 
If $A=\zen{3}{A}$, then $A^2$ is abelian.
\end{thm}
 
We shall let binomial coefficients be denoted $\begin{pmatrix} m\\ i\end{pmatrix} $ as usual.
Next we let $n$ be a fixed positive integer and let $\mathbb{Z}_n=\mathbb{Z}/n\mathbb{Z}$ denote the ring of integers modulo $n$. Let 
\[
D(1,3)=\mathbb{Z}_n\times\mathbb{Z}_n\times\mathbb{Z}_n\times\mathbb{Z}_n
\]
where addition and multiplication are defined by
\begin{align*}
(m_1, m_2,m_3,m_4)+(n_1,n_2,n_3,n_4)&=(k_1, k_2,k_3,k_4)
\end{align*}
and
\begin{align*}
&(m_1, m_2,m_3,m_4)\cdot (n_1,n_2,n_3,n_4)\\
&=(k_1,k_2+m_1n_1,
k_3+m_1n_2+m_2n_1,k_4+m_2n_1-\begin{pmatrix} m_1\\ 2\end{pmatrix} n_1)
\end{align*}
where $k_i=m_i+n_i$ and $m_i,n_i\in \mathbb{Z}_n$ for $i=1,\dots,4$.

Concerning $D(1,3)$ we shall prove the following result.

\begin{thm} \label{t:mainc}
\begin{enumerate}
\item[ (i)] The set $D(1,3)$ with the addition and multiplication as defined is a one-generator, Smoktunowicz-nilpotent left brace and $\zl{D(1,3)}=3$.  Furthermore the orders of the elements of the additive group of $D(1,3)$ are divisors of $n$ and $D(1,3)^2$ is abelian.
\item[(ii)] Let $A$ be a non-abelian one-generator left brace whose generator $a$ satisfies $na=0$ for some positive integer $n$.  If $A=\zen{3}{A}$,  
then $A$ is an epimorphic image of  the left brace $D(1,3)$.
\end{enumerate}
\end{thm}


 
 \section{Preliminary Results}\label{s:prelim}
 
 We shall need some of the following properties of $\star$ and $\lambda_a$ whose proofs can be found in \cite{fC18} or \cite{JVV22}.

\begin{lemma}\label{l:lem2}
Let $A$ be a left brace.  Then
\begin{enumerate}
\item[  (i)] $a\star (b+c)=a\star b+ a\star c$;
\item[ (ii)] $(ab)\star c=a\star (b\star c)+b\star c+ a\star c$;
\item[(iii)] $(a+b)\star c=a\star (\lambda_{a^{-1}}(b)\star c) +(\lambda_{a^{-1}}(b)\star c)+a\star c$;
\item[ (iv)] $\lambda_y(b\star a)=yby^{-1}\star \lambda_y(a)$;
\item[  (v)] $yby^{-1}=\lambda_y(\lambda_b(y^{-1}) -y^{-1} +b)=\lambda_y(b\star y^{-1} +b)$
\end{enumerate}
for all elements $a,b,c,y\in A$.
\end{lemma}

We shall also require the following results, the first of which is well-known and its proof is omitted.

\begin{lemma}\label{l:lem1}
Let $n,k$ be natural numbers and let $a, b_1,b_2,\dots, b_n, c_1,c_2,\dots, c_k$ be elements of the  left brace $A$.  Then 
\[
a(b_1+\dots+ b_n-c_1-\dots - c_k)=ab_1+ab_2+\dots +ab_n-ac_1-\dots -ac_k +(k-n+1)a.
\]
\end{lemma}

The next two results have appeared in \cite{BEKP24} but we include the proofs for the sake of completeness.  We observe the well-known fact that $\begin{pmatrix} n+1\\ i+1\end{pmatrix}=\begin{pmatrix} n\\ i\end{pmatrix}+ \begin{pmatrix} n\\ i+1\end{pmatrix}$. 

\begin{proposition}\label{p:prop1}
Let $A$ be a left brace and let $a=a_1\in A$. Let $a_{n+1}=a_1\star a_n$ for all natural numbers $n$. Then
\[
a_1^n=\begin{pmatrix} n\\ 1\end{pmatrix}a_1+ \begin{pmatrix} n\\ 2\end{pmatrix}a_2+\dots +\begin{pmatrix} n\\ n-1\end{pmatrix}a_{n-1}+a_n
\]
for all natural numbers $n$.
\end{proposition}

\begin{proof}
First we note that for all natural numbers $n$ we have $a_{n+1}=a_1\star a_n=a_1a_n-a_1-a_n$. Hence $a_1a_n=a_1+a_n+a_{n+1}$ for all such $n$. We use induction on $n$ and observe that the case $n=1$ is clear.  We assume the natural  induction hypothesis. Then we have, using the induction hypothesis, Lemma~\ref{l:lem1}  and the remarks above,
\begin{align*}
a_1^{n+1}& =a_1\cdot a_1^n =a_1\left(\begin{pmatrix} n\\ 1\end{pmatrix}a_1+\begin{pmatrix} n\\ 2\end{pmatrix}a_2+\dots +\begin{pmatrix} n\\ n-1\end{pmatrix}a_{n-1}+a_n\right)\\
&= \begin{pmatrix} n\\ 1\end{pmatrix}a_1a_1 +\begin{pmatrix} n\\ 2\end{pmatrix}a_1a_2+\dots + \begin{pmatrix} n\\ n-1\end{pmatrix}a_1a_{n-1}+a_1a_n\\
&-\left(\begin{pmatrix} n\\ 1\end{pmatrix}+\dots+\begin{pmatrix} n\\ n-1\end{pmatrix}\right)a_1\\
&=\begin{pmatrix} n\\ 1\end{pmatrix}a_1+\begin{pmatrix} n\\ 1\end{pmatrix}a_2+\begin{pmatrix} n\\ 2\end{pmatrix}a_2+\dots +\begin{pmatrix} n\\ n-2\end{pmatrix}a_{n-1}+\begin{pmatrix} n\\ n-1\end{pmatrix}a_{n-1}\\
&+\begin{pmatrix} n\\ n-1\end{pmatrix}a_n+a_1+a_n+a_{n+1}\\
&=\begin{pmatrix} n+1\\ 1\end{pmatrix}a_1+\begin{pmatrix} n+1\\ 2\end{pmatrix}a_2+\dots +\begin{pmatrix} n+1\\ n\end{pmatrix}a_n+a_{n+1}.
\end{align*}
This proves the result by induction.
\end{proof}

\begin{proposition}\label{p:prop2}
Let $A$ be a left brace and let $a=a_1$ be an arbitrary element of $A$.  Let $a_{n+1}=a_1\star a_n$
for all natural numbers $n$. Then
\[
a_1^n\star a_1= \begin{pmatrix} n\\ 1\end{pmatrix}a_2+ \begin{pmatrix} n\\ 2\end{pmatrix}a_3+\dots +\begin{pmatrix} n\\ n-1\end{pmatrix}a_{n}+a_{n+1}.
\]
\end{proposition}

\begin{proof}
Using  Proposition~\ref{p:prop1} 
we have
\begin{align*}
a_1^n\star a_1&=a_1^n\cdot a_1-a_1^n-a_1\\
&=\begin{pmatrix} n+1\\ 1\end{pmatrix}a_1+ \begin{pmatrix} n+1\\ 2\end{pmatrix}a_2+\dots + \begin{pmatrix} n+1\\ n\end{pmatrix}a_n+ a_{n+1}\\
&-\begin{pmatrix} n\\ 1\end{pmatrix}a_1-\begin{pmatrix} n\\ 2\end{pmatrix}a_2-\dots -\begin{pmatrix} n\\ n-1\end{pmatrix}a_{n-1}-a_n-a_1\\
&=\left(\begin{pmatrix} n+1\\ 1\end{pmatrix}-\begin{pmatrix} n\\ 1\end{pmatrix}-1\right)a_1+\left(\begin{pmatrix} n+1\\ 2\end{pmatrix}-\begin{pmatrix} n\\ 2\end{pmatrix}\right)a_2+\dots\\
&+ \left(\begin{pmatrix} n+1\\ n\end{pmatrix}-\begin{pmatrix} n\\ n\end{pmatrix}\right)a_n+a_{n+1}\\
&=\begin{pmatrix} n\\ 1\end{pmatrix}a_2+\dots +\begin{pmatrix} n\\ n-1\end{pmatrix}a_n+a_{n+1}.
\end{align*}
This concludes the proof.
\end{proof}

\begin{proposition}\label{p:prop4}
The set $D(1,2)$ with the addition and multiplication as defined is a one-generator, Smoktunowicz-nilpotent left brace and $\zl{D(1,2)}=2$.
\end{proposition}

\begin{proof}
We refer to \eqref{eq:defD12}.  Since the addition in $D(1,2)$ is the usual addition of $\mathbb{Z}\times \mathbb{Z}$ it is clear that $D(1,2)$ is an abelian group under addition with additive identity $(0,0)$.  Next we observe that 
\[
(0,0)(n_1,n_2)=(0+n_1, 0+n_2+0\cdot n_2)=(n_1,n_2)=(n_1,n_2)(0,0)
\]
so $(0,0)$ is a multiplicative identity. 

It is easy to see that the associative law for multiplication holds in $D(1,2)$ and we omit this computation. Furthermore $(n_1,n_2)^{-1}=(-n_1,n_1^2-n_2)$ since
\begin{align*}
(n_1,n_2)(-n_1,n_1^2-n_2)&=(n_1-n_1,n_2 +n_1^2-n_2-n_1^2)\\
&=(0,0)=(-n_1,n_1^2-n_2)(n_1,n_2).
\end{align*}
Hence $D(1,2)$ is a group under multiplication. Indeed the multiplication is commutative also, by \eqref{eq:defD12}.

Finally we note that $D(1,2)$ is a left brace since
\[
(l_1,l_2)((m_1,m_2)+(n_1,n_2))=(l_1+m_1+n_1,l_2+m_2+n_2+l_1(m_1+n_1))
\]
whereas 
\begin{align*}
&(l_1,l_2)(m_1,m_2)+(l_1,l_2)(n_1,n_2)-(l_1,l_2) \\
&=(l_1+m_1,l_2+m_2+m_1l_1) + (l_1+n_1,l_2+n_2+l_1n_1)-(l_1,l_2)\\
&=(l_1+m_1+n_1,l_2+m_2+n_2+l_1m_1+l_1n_1)
\end{align*}
so (LB3) is satisfied.

It also follows that 
\begin{align*}
&(l_1,l_2)\star (m_1,m_2)=(l_1,l_2)(m_1,m_2)-(l_1,l_2)-(m_1,m_2)\\
&= (l_1+m_1,l_2+m_2+l_1m_1)-(l_1,l_2)-(m_1,m_2)=(0,l_1m_1).
\end{align*}

Let $E=\{(0,k)|k\in\mathbb{Z}\}$.  Then $D(1,2)^2\leq E$ and also $(n_1,n_2)\star (0,k)=(0,0)$. Hence $E\leq \zet{D(1,2)}$ and so $D(1,2)$ is $\star$-nilpotent.  Indeed $\zl{D(1,2)}=2$.  

Finally $(1,0)\star (1,0)=(0,1)$ so
\[
 (n_1,n_2)=(n_1,0)+(0,n_2)=n_1(1,0)+n_2((1,0)\star (1,0))
\]
which shows that $D(1,2)$ is generated as a brace by the element $(1,0)$.
\end{proof}

\section{The Proof of Theorem~\ref{t:maina}}

The first part of Theorem~\ref{t:maina} follows from Proposition~\ref{p:prop4} so we only need to prove the second part.

\begin{proof}
Let $A$ be generated as a brace by $a$ and let  $c=a\star a$. We note that $c\in \zet{A}$. It follows that $a\star c=c\star a=0$ and Proposition~\ref{p:prop1} implies that for all positive integers $k$ we have 
\[
a^k=ka+\dfrac{k(k-1)}2 c=(ka)\left(\dfrac{k(k-1)}2 c\right)
\]
 since $\dfrac{k(k-1)}2 c\in \zet{A}$.  It is clear that $a^k\star a=a\star a^k$ so we have, using Lemma~\ref{l:lem2},  that 
\[
a^k\star a=a\star \left(ka+ \dfrac{k(k-1)}2 c\right) =a\star ka =k(a\star a)=kc.
\]
Also, again using Lemma~\ref{l:lem2},  
\begin{align*}
a^k\star a &= \left((ka)\left(\dfrac{k(k-1)}2 c\right)\right)\star a=ka\star \left(\dfrac{k(k-1)}2 c \star a\right) +\left(\dfrac{k(k-1)}2 c \star a\right)\\
& + ka\star a
=ka\star a
\end{align*}
so that for all positive integers $k$ we have 
\begin{equation}\label{eq:eqthm1}
a^k\star a=kc=ka\star a.
\end{equation}

Next we note that $a^{-k}a^k\star a=0$. Using Lemma~\ref{l:lem2} and \eqref{eq:eqthm1} we deduce that 
\[
0=a^{-k}\star (a^k\star a)+(a^k\star a)+ (a^{-k}\star a).
\]
Since $a^k\star a=kc\in \zet{A}$ this implies that $a^{-k}\star a=-kc=-k(a\star a)$. Since $A=\zen{2}{A}$ we have $xy\zet{A}=(x+y)\zet{A}$ for all $x,y\in A$ and in particular $a^{-k}\zet{A}=(-ka)\zet{A}$.  Hence there exists $u\in \zet{A}$ such that 
\[
a^{-k}\star a=((-ka)u)\star a=(-ka)\star a,
\]
using Lemma~\ref{l:lem2}
Thus
\begin{equation}\label{eq:thm1b}
a^n\star a=na\star a=nc \text{ for all integers }n.
\end{equation}

Let $B=\{\al a+ \ga c|\al, \ga\in \mathbb{Z}\}$ and let $x=\al a +\ga c, y=\al_1a+\ga_1 c$ where $\al, \ga, \al_1, \ga_1\in \mathbb{Z}$.  We note that $B$ is clearly an additive subgroup of $A$.  Using Lemma~\ref{l:lem2} and the facts assembled above we see that 
\begin{align*}
x\star  y &=(\al a +\ga c)\star (\al_1a+\ga_1 c)= (\al a +\ga c)\star (\al_1 a)+(\al a +\ga c)\star \ga_1 c\\
&=(\al a +\ga c)\star (\al_1 a) =\al_1((\al a +\ga c)\star a)\\
&=\al_1 ((\al a)(\ga c)\star a)=\al_1 ((\al a)\star a)=\al_1\al c.
\end{align*}
It follows that 
\begin{align}\label{eq:thm1c}
xy &=x\star y +x+y 
=(\al +\al_1)a+ (\ga+\ga_1+\al\al_1)c\in B. 
\end{align}
Furthermore,  it is easy to verify using \eqref{eq:thm1c} that  $x^{-1} =-\al a +(\al ^2-\ga)c$ so that $x^{-1}\in B$.  This means that $B$ is a subbrace of $A$ and that $A=\br{a}\leq B$.  Hence $A=B$.

We now define a mapping $f:D(1,2)\longrightarrow A$ by $f(\al, \ga)=\al a+\ga c$ for each $\al,\ga\in \mathbb{Z}$. The various definitions and \eqref{eq:thm1c} imply that $f$ is a homomorphism of left braces and since $A=B$ the map $f$ is an epimorphism.
\end{proof}

We move next to the brace $D(1,3)$, showing that it is indeed a left brace and that it is Smoktunowicz nilpotent.

\begin{proposition}\label{p:prop5}
The set $D(1,3)$ with the addition and multiplication as defined is a one-generator, Smoktunowicz-nilpotent left brace and $\zl{D(1,3)}=3$.  Furthermore the orders of the elements of the additive group of $D(1,3)$ are divisors of $n$ and $D(1,3)^2$ is abelian.
\end{proposition}

\begin{proof}
It is clear that $D(1,3)$ is an abelian group under addition.  Furthermore, it is easy to see that 
$(0,0,0,0)$ will act as a multiplicative identity.  It is an easy, although slightly tedious,  computation to show that $D(1,3)$ has an associative multiplicative operation. Furthermore we have 
\begin{align*}
&\oq{x}(-x_1, x_1^2-x_2, 2x_1x_2-x_3-x_1^3,x_1x_2-x_4-\begin{pmatrix} x_1\\2\end{pmatrix} x_1)\\
&=(x_1-x_1, x_2+x_1^2-x_2+x_1(-x_1), x_3+2x_1x_2-x_3-x_1^3+x_1(x_1^2-x_2)+x_2(-x_1),\\& x_4+x_1x_2-x_4-\begin{pmatrix} x_1\\2\end{pmatrix} x_1+x_2(-x_1)-\begin{pmatrix} x_1\\ 2\end{pmatrix}(-x_1))\\
&=(0,0,0,0).
\end{align*}
Thus 
\begin{align*}\oq{x}^{-1}=(-x_1, x_1^2-x_2, 2x_1x_2-x_3-x_1^3,x_1x_2-x_4-\begin{pmatrix} x_1\\2\end{pmatrix} x_1)
\end{align*}
so that $D(1,3)$ is also a group under multiplication, although it is not abelian since in general the fourth component of $\oq{x}\oq{y}$ is different from the fourth component of $\oq{y}\oq{x}$. If $x=\oq{x}, y=\oq{y}, z=\oq{z}$, then we have 
\begin{align*}
&x(y+z) \\
=&x(y_1+z_1, y_2+z_2, y_3+z_3, y_4+z_4)\\
=&(x_1+y_1+z_1, x_2+y_2+z_2+x_1(y_1+z_1), \\
&x_3+y_3+z_3+x_1(y_2+z_2)+x_2(y_1+z_1), \\
 & x_4+y_4+z_4+x_2(y_1+z_1)-\begin{pmatrix} x_1\\ 2\end{pmatrix}(y_1+z_1))
\end{align*}
whereas
\begin{align*}
&xy+xz-x\\
=&(x_1+y_1,x_2+y_2+x_1y_1,x_3+y_3+x_1y_2+x_2y_1,\\
&x_4+y_4+x_2y_1-\begin{pmatrix} x_1\\ 2\end{pmatrix}y_1) \\
+& (x_1+z_1,x_2+z_2+x_1z_1,x_3+z_3+x_1z_2+x_2z_1,\\
&x_4+z_4+x_2z_1-\begin{pmatrix} x_1\\ 2\end{pmatrix}z_1) \\
&-\oq{x}
\end{align*}
and inspection of the two final expressions clearly yields that $x(y+z)=xy+xz-x$ so that $D(1,3)$ is a left brace.

A further easy computation shows that 
\begin{equation}\label{eq:prop5}
x\star y =(0,x_1y_1, x_1y_2+x_2y_1, x_2y_1-\begin{pmatrix} x_1\\ 2\end{pmatrix}y_1).
\end{equation}
Let $D_2=\{(0,x_2,x_3,x_4)| x_i\in\mathbb{Z}\}$ and $D_1=\{(0,0,x_3,x_4)|x_i\in\mathbb{Z}\}$.  These are clearly subbraces of $D(1,3)$ using the work above,  and our remark above shows that $D(1,3)^2\leq D_2$. Furthermore, if $x=\oq{x}\in\zet{D(1,3)}$, then  $x\star y=y\star x=0$ for all $y\in D(1,3)$ and it is straightforward to see that in this case $x_1=x_2=0$, whereas $x_3, x_4$ can be arbitrary. It follows that $D_1=\zet{D(1,3)}$.  Also if $\oq{x}\in D_2$, then $x_1=0$ and 
\begin{align*}
&\oq{x}\star \oq{y} =(0,0,x_2y_1,x_2y_1)\in D_1\\
& \oq{y}\star \oq{x} =(0,0,x_1y_2,x_1y_2) \in D_1
\end{align*}
so $D_2\leq \zen{2}{D(1,3)}$ and indeed this shows that $D(1,3)$ is Smoktunowicz-nilpotent  with $\zl{D(1,3)}=3$.  Since 
\[(0,x_2,x_3,x_4)(0,y_2,y_3,y_4)=(0,x_2+y_2,x_3+y_3,x_4+y_4),
\]
 $D_2$ is an abelian brace so $D(1,3)^2$ is abelian. Also 
 \begin{align*}
 (1,0,0,0)\star (1,0,0,0)&=(0,1,0,0),\\
 (1,0,0,0)\star (0,1,0,0)&=(0,0,1,0),\\
 (0,1,0,0)\star (1,0,0,0)&=(0,0,1,1),
 \end{align*}
 so that $D(1,3)$ is generated as a brace by $(1,0,0,0)$.
 
 Finally,  it is clear that $|D(1,3)|=n^4$ and that $nx=(0,0,0,0)$ for all $x\in D(1,3)$.

\end{proof}

\section{The Proofs of Theorems~\ref{t:mainb} and \ref{t:mainc}}\label{s:thmb}
The proofs of Theorem~\ref{t:mainb} and \ref{t:mainc} have some common themes and we shall employ some common notation which we introduce here.  In both results we assume that $A=\zen{3}{A}$ is a nonabelian left brace generated by an element $a$. For now we do not need to assume $na=0$.

We set $b=a\star a\in\zen{2}{A}=C_2$ and  $c=b\star b$.  We note that since $\zl{A}=3$ then $c_1=a\star b, c_2= b\star a, c\in \zen{1}{A}=C_1$. 
Then 
we have $c_1=ab-a-b$ and $c_2=ba-b-a$ so 
\begin{equation}\label{eq:pfb1}
z=c_2-c_1=ba-ab\in C_1.
\end{equation}

 Also we note that $ab=a+b+c_1$ and $a^2=2a+b$.


\begin{lemma}\label{l:newlem1}
With the notation established above,  if $na=0$,  then $nb=nc_1=nc_2=nc=0$.
\end{lemma}

\begin{proof}
Using Lemma~\ref{l:lem2} we have $nb=n(a\star a)=a\star na=a\star 0=0$.  The rest of the lemma follows in a similar fashion using the definitions of these elements.
\end{proof}

\begin{lemma}\label{l:newlem2}
With the notation established above, for all natural numbers $k, s$:
\begin{enumerate}
\item[  (i)] $a^k=ka +\begin{pmatrix} k\\2 \end{pmatrix}b + \begin{pmatrix} k\\3 \end{pmatrix}c_1$, for $k\geq 3$;
\item[ (ii)] $ka=a^k-\begin{pmatrix} k\\2 \end{pmatrix}b - \begin{pmatrix} k\\3 \end{pmatrix}c_1 = 
(a^k-\begin{pmatrix} k\\2 \end{pmatrix}b)(-\begin{pmatrix} k\\3 \end{pmatrix}c_1)$;
\item[(iii)] $a^k\star a=kb + \begin{pmatrix} k\\2 \end{pmatrix}c_1$;
\item[ (iv)] $a^k\star b= kc_1$;
\item[  (v)] $b^k\star a=kc_1+kz=kb\star a=kc_2$;
\item[ (vi)] $kb\star b=b^k\star b=k(b\star b)$. In particular, if $b\star b=0$, then $kb\star b=0$.
\item[(vii)] $ka+sb=wa^k\left(s-\begin{pmatrix} k\\2\end{pmatrix}\right)b$ for some element $w\in C_1$.
\end{enumerate}
\end{lemma}

\begin{proof}
Since $A=\zen{3}{A}$, Proposition~\ref{p:prop1} implies (i) immediately. The first part of (ii) is just a rewriting of this and the second part of (ii) follows  because $\begin{pmatrix} k\\3 \end{pmatrix}c_1 \in C_1$.  Part (iii) easily follows using  Proposition~\ref{p:prop2}.  For part (iv) we use induction. 
We already know that $a\star b=c_1$ by definition and we suppose inductively that $a^k\star b=kc_1$. Then, using Lemma~\ref{l:lem2} we have 
\[
a^{k+1}\star b=(aa^k)\star b=a\star (a^k\star b) +(a^k\star b)+ (a\star b)=kc_1+c_1=(k+1)c_1
\]
using the induction hypothesis and the fact that $a^k\star b\in \zet{A}$.

For part (v) we note that $kb+C_1=b^k+C_1=b^kC_1$ since $b\in\zen{2}{A}$ and hence $kb=w_1b^k$ for some element $w_1\in C_1$.  It then follows that $kb\star a=w_1b^k\star a=b^k\star a$, using Lemma~\ref{l:lem2}.  We already know that $b\star a=c_2=c_1+z$ by equation~\eqref{eq:pfb1}. Inductively suppose that $b^k\star a=kc_1+kz$ for some natural number $k$.  Then 
\[
b^{k+1}\star a= (bb^k)\star a=b\star (b^k\star a)+b^k\star a +b\star a=kc_1+kz+c_1+z
\]
using the induction hypothesis and Lemma~\ref{l:lem2} which proves (v).

Since  $kb\star b=b^k\star b$ a further  easy induction implies (vi).  Finally,  to prove (vii) we have 
\begin{align*}
ka+sb+C_1& =a^k-\begin{pmatrix} k\\2 \end{pmatrix}b - \begin{pmatrix} k\\3 \end{pmatrix}c_1+sb +C_1= 
a^k-\begin{pmatrix} k\\2 \end{pmatrix}b+sb+C_1 \\
& 
=a^k(\left(s-\begin{pmatrix} k\\2 \end{pmatrix}\right)b) +C_1 
=a^k(\left(s-\begin{pmatrix} k\\2 \end{pmatrix}\right)b) C_1
\end{align*}
using (ii) and the fact that $b\in \zen{2}{A}$. Hence $ka+sb=w a^k(\left(s-\begin{pmatrix} k\\2 \end{pmatrix}\right)b)$ for some element $w\in C_1$.
\end{proof}

We next observe that in our situation we have $b\star b=0$. Indeed, our result is slightly more general.

\begin{lemma}\label{l:newlem3}
With the notation established above let $A=\br{a}$ be a one-generator Smoktunowicz-nilpotent left brace  such that $\zl{A}=3$. If $b=a\star a$, then 
$b\star b=0$.
\end{lemma}

\begin{proof}
Using Lemma~\ref{l:lem2} and our established notation we have,  
\begin{align}\label{eq:t18a}
ba^2 & =b\star a^2 + b+a^2= b\star (2a+b) +2a+2b  \\
&= 2(b\star a) +b\star b +2a+2b 
= 2c_2 +2a+2b+b\star b \notag\\
&=2c_2 +2a+2b+c.
\end{align}

On the other hand $ba=b\star a +b +a=c_2+b+a$. 
Then $ba+C_1=c_2+b+a+C_1=b+a+C_1=ab+C_1
=abC_1$ since $c\in C_1$ and $b\in C_2$.  Thus $ba= c_2+b+a=wab$
for some $w\in C_1$. Using Lemma~\ref{l:lem2} and the fact that $w,  c\in C_1$ we have 
\begin{equation*}\label{eq:thm18b}
ba\star a= (wab)\star a=(ab)\star a= a\star (b\star a)+b\star a+ a\star a=c_2+b
\end{equation*}
and hence 
\begin{equation}\label{eq:thm18c}
(ba)a=ba\star a +ba+a=(c_2+b)+(c_2+b+a)+a=2c_2+2b+2a.
\end{equation}
Comparing \eqref{eq:t18a} and \eqref{eq:thm18c}  we see that $c=b\star b=0$ since $(ba)a=b(a^2)$.
\end{proof}

\begin{proof}[\textbf{Proof of Theorem~\ref{t:mainb}}] 

Clearly we may assume that $A=\br{a}$ is non-abelian.  We shall use the notation established at the start of this section and we shall now assume that $na=0$.

Let 
\[
B=\{k_1 a+k_2 b +k_3 c_1+ k_4 c_2| k_i \in\mathbb{Z}_n \text{ for }i=1,\dots, 4\}.
\]
It is clear that $B$ is an additive subgroup of the brace $A$.
We let
\[
x=n_1 a+ n_2 b+ n_3c_1+n_4c_2, y=t_1 a+ t_2 b+ t_3c_1+t_4c_2
\]
where $n_i, t_i\in\mathbb{Z}_n$ for $i=1,\dots 4$.  

It follows from Lemmas~\ref{l:lem2}, \ref{l:newlem1}, \ref{l:newlem2} and the fact that $
c_1,c_2\in\zet{A}$ that
\[
x\star y=(n_1a+n_2b)\star (t_1a+t_2b)=t_1((n_1a+n_2b)\star a)+ t_2((n_1a+n_2b)\star b).
\]
Now by Lemma~\ref{l:newlem2} there exists $w\in C_1$ such that
\begin{align*}
(n_1a+n_2b)\star a &= \left(wa^{n_1}\left(n_2-\begin{pmatrix} n_1\\2\end{pmatrix}\right)b\right)\star a =\left(a^{n_1}\left(n_2-\begin{pmatrix} n_1\\2\end{pmatrix}\right)b\right)\star a \\
&=a^{n_1} \star (\left(n_2-\begin{pmatrix} n_1\\2\end{pmatrix}\right)b \star a) + \left(n_2-\begin{pmatrix} n_1\\2\end{pmatrix}\right)b \star a+(a^{n_1}\star a) \\
&= n_1b+ \begin{pmatrix} n_1\\2\end{pmatrix} c_1+\left(n_2-\begin{pmatrix} n_1\\2\end{pmatrix}\right)c_2.
\end{align*}

Furthermore, 
\begin{align*}
(n_1a+n_2b)\star b &= \left(wa^{n_1}\left(n_2-\begin{pmatrix} n_1\\2\end{pmatrix}\right)b\right) \star b =\left(a^{n_1}\left(n_2-\begin{pmatrix} n_1\\2\end{pmatrix}\right)b\right)\star b\\
&=a^{n_1}\star (\left(n_2-\begin{pmatrix} n_1\\2\end{pmatrix}\right)b\star b) + \left(n_2-\begin{pmatrix} n_1\\2\end{pmatrix}\right)b\star b + (a^{n_1}\star b)\\
&=
n_1c_1+ \left(n_2-\begin{pmatrix} n_1\\2\end{pmatrix}\right)c=n_1c_1.
\end{align*}
It follows easily that 
\begin{equation}\label{eq:pfz}
x\star y= t_1n_1b+ (t_1\begin{pmatrix} n_1\\2\end{pmatrix} +t_2n_1)c_1+ t_1\left(n_2-\begin{pmatrix} n_1\\2\end{pmatrix}\right)c_2 \in B.
\end{equation}
Since $xy=x\star y+ x+ y$ it follows that $xy \in B$.  Also it is easy to see that 
\[
x^{-1}=-n_1a +(n_1^2-n_2)b+(2n_1n_2-n_3-n_1^3)c_1+\left(n_1n_2-n_3-n_1\begin{pmatrix} n_1\\2 \end{pmatrix}\right)z
\]
which is visibly an element of $B$ since $z=c_2-c_1$.  It follows that $B$ is a subbrace of $A$ and since $a\in B$ it is easy to see that $B=\br{a}=A$. 

Now it follows from \eqref{eq:pfz} that the elements of $A^2$ are precisely those elements in which the coefficient of $a$ is $0$. Thus if $x,y\in A^2$, then $n_1=t_1=0$ and it follows again using \eqref{eq:pfz} that in this case $x\star y=0$. It follows that $A^2=A^{(2)}$ is abelian as required.
\end{proof}

\begin{proof}[\textbf{Proof of Theorem~\ref{t:mainc}}] 

The first part of Theorem~\ref{t:mainc} was proved in Proposition~\ref{p:prop5} so we only need to prove the second part. It follows from Theorem~\ref{t:mainb} that $A^{(2)}$ is abelian. Furthermore, we showed in the proof of Theorem~\ref{t:mainb} that the additive subgroup $B$ was precisely the subbrace generated by $a$ and hence $B=\br{a}=A$.  We then define a map $f:D(1,3)\longrightarrow A$ by
\[
f((n_1,n_2,n_3,n_4))= n_1a+n_2b+n_3c_1+n_4z
\]
where $n_i\in\mathbb{Z}_n$ for $i=1,\dots, 4$.

The equalities above show that the function $f$ is a homomorphism of left braces and since $A=B$ the map $f$ is a surjection. This completes the proof.
\end{proof}



We should like to thank the referee for their astute observations.


\providecommand{\bysame}{\leavevmode\hbox to3em{\hrulefill}\thinspace}
\providecommand{\MR}{\relax\ifhmode\unskip\space\fi MR }
\providecommand{\MRhref}[2]{%
  \href{http://www.ams.org/mathscinet-getitem?mr=#1}{#2}
}
\providecommand{\href}[2]{#2}

\end{document}